\documentclass[11pt]{article}
\usepackage{qr}

\renewcommand{\gets}{\leftarrow}
\usepackage{boxedminipage}

\newcommand{\iqr}{\mathsf{iqr}}

\newcommand{\apr}{\mathrm{OptRitz}}

\renewcommand{\r}{\theta}

\newcommand{\K}{B}

\newcommand{\tapr}{T_{\mathrm{\apr}}}

\newcommand{\pot}{\psi_k}

\newcommand{\exactqr}{\mathsf{iqr}}
\newcommand{\find}{\mathrm{Find}}
\newcommand{\exc}{\mathrm{Exc}}
\newcommand{\tnet}{N_{\mathsf{net}}}

\newcommand{\corner}[2]{#1_{(#2)}}

\newcommand{\E}{\mathbb{E}}
\renewcommand{\P}{\mathbb{P}}

\newcommand{\acc}{\delta}

\newcommand{\cp}{\alpha}
\renewcommand{\next}[1]{\widehat{#1}}

\renewcommand{\bC}{\mathbb{C}}

\theoremstyle{plain}
\newtheorem{qd}{Question}

\title{Global Convergence of Hessenberg Shifted QR I: Exact Arithmetic}
\author{Jess Banks\thanks{\texttt{jess.m.banks@berkeley.edu}. Supported by NSF GRFP Grant DGE-1752814 and NSF Grant  CCF-2009011.}\\ UC Berkeley \and  Jorge Garza-Vargas\thanks{\texttt{jgarzavargas@berkeley.edu}. Supported by NSF Grant  CCF-2009011.}\\ UC Berkeley \and  Nikhil Srivastava\thanks{\texttt{nikhil@math.berkeley.edu}. Supported by NSF Grant  CCF-2009011.}  \\    UC Berkeley }
\date{\today}

\begin{document}
\maketitle

\begin{abstract}
Rapid convergence of the shifted QR algorithm on symmetric matrices was shown
	more than fifty years ago. Since then, despite significant interest and
	its practical relevance, an understanding of the dynamics and convergence properties of the shifted QR algorithm on nonsymmetric matrices has remained elusive.
  
We introduce a new family of shifting strategies for the Hessenberg shifted QR algorithm. We prove that when the input is a diagonalizable Hessenberg matrix $H$ of bounded \emph{eigenvector condition number} $\kappa_V(H)$ ---  defined as the minimum condition number of $V$ over all diagonalizations $VDV^{-1}$ of $H$ --- then the shifted QR algorithm with a certain strategy from our family is guaranteed to converge rapidly to a  Hessenberg matrix with a zero subdiagonal entry, in exact arithmetic. Our convergence result is nonasymptotic, showing that the geometric mean of certain subdiagonal entries of $H$ decays by a fixed constant in every $QR$ iteration. The  arithmetic cost of implementing each iteration of our strategy scales roughly logarithmically in the eigenvector condition number $\kappa_V(H)$, which is a measure of the nonnormality of $H$.
 
 The key ideas in the design and analysis of our strategy are: (1) We  are able to precisely characterize when a certain shifting strategy based on Ritz values stagnates. We use this information to design  certain ``exceptional shifts'' which are guaranteed to escape stagnation whenever it occurs. (2) We use higher degree shifts (of degree roughly $\log \kappa_V(H)$) to dampen transient effects due to nonnormality, allowing us to treat nonnormal matrices in a manner similar to normal matrices.
  
 \end{abstract}

\tableofcontents
\newcommand{\Pol}{\mathcal{P}}
\newcommand{\Sh}{\mathrm{Sh}}
\renewcommand{\H}{\mathbb{H}}
\section{Introduction}

The Hessenberg shifted QR algorithm, discovered in the late 1950's by Francis \cite{francis1961qr,francis1962qr} (see also Kublanovskaya \cite{kublanovskaya1962some}), is and has been for several decades the most widely used
method for approximately\footnote{In the sense of backward error, i.e., exactly computing the eigenvalues of a nearby matrix.} computing all of the eigenvalues of a dense matrix. It is implemented in all of the major software packages for numerical linear algebra and was listed as one of the ``Top 10 algorithms of the twentieth century'' %along with the Metropolis algorithm and the Simplex algorithm 
\cite{dongarra2000guest, parlett2000qr}. The algorithm is specified by a {\em shifting strategy},
which is an efficiently computable\footnote{In this paper, we assume exact arithmetic with complex numbers and count arithmetic operations as a measure of complexity.} function $$\Sh:\H^{n\times n}\rightarrow \Pol_k,$$ where $\H^{n\times n}$ is the set of $n\times n$ complex Hessenberg\footnote{A matrix $H$ is (upper) Hessenberg  if $H(i,j)=0$ whenever $i>j+1$. Such matrices are ``almost'' upper triangular.} matrices and $\Pol_k$ is the set of monic complex univariate
polynomials of degree $k$, for some $k=k(n)$ typically much smaller than $n$. The word ``shift'' comes from the fact that when $k=1$ we have $p_t(H_t)=H_t-s_t$\footnote{Hereon, following a convention from operator theory, we will use scalars to denote scalar multiples of the identity.} for some $s_t\in \C$. The algorithm then consists of the following discrete-time
isospectral nonlinear dynamical system on $\H^{n\times n}$, given an initial condition $H_0$:
\begin{align}
    Q_tR_t &= p_t(H_t)  \quad&\text{where }p_t=\Sh(H_t),\label{eqn:iterdef}\\
 H_{t+1} &= Q_t^*H_tQ_t,\quad& t=0,1,2,\ldots \nonumber.
\end{align} 
The first step in \eqref{eqn:iterdef} is a $QR$ decomposition, i.e.  $Q_t$ is unitary and $R_t$ is upper triangular. It is  not hard to see that this iteration preserves Hessenberg structure. Moreover, under the specification that the diagonal entries of $R_t$ are positive, the QR decomposition is uniquely determined whenever $p_t(H_t)$ is invertible. Therefore, for the sake of simplicity, in this introduction  we will  ignore the case when $p_t(H_t)$ is singular (see Section \ref{sec:related} for a discussion).

The relevance of this iteration to the eigenvalue problem stems from two facts. First, every matrix $A\in\C^{n\times n}$ is unitarily similar to a Hessenberg matrix $H_0$, and in exact arithmetic such a similarity can be computed exactly in $O(n^3)$ operations. Second, it was shown in \cite{francis1961qr, kublanovskaya1962some} that for the trivial ``unshifted''  strategy $p(z)=z$, in the generic situation when $H_0$ has distinct moduli, the iterates $H_t$  always converge to an upper triangular matrix $H_\infty$; this is because iterates of the unshifted QR algorithm are in one-to-one correspondence with the iterates generated by running simultaneous iteration (a specific form of subspace iteration) for $H_0$ (or equivalently inverse iteration for $H_0^*$) on the canonical orthonormal basis; we refer the reader to  \cite{watkins1982understanding} for details. Combining the unitary similarities accumulated during the iteration, these two facts yield a Schur factorization
$A=Q^*H_\infty Q$
of the original matrix, from which the eigenvalues of $A$ can be read off. The unshifted QR iteration does {\em not} give an efficient algorithm, however, as it is easy to see that convergence can be arbitrarily slow if the ratios of the magnitudes of the eigenvalues of $H_0$ are close to $1$. The role of the shifting strategy is to adaptively improve these ratios and thereby accelerate convergence. The challenge is that this must be done efficiently without prior knowledge of the eigenvalues.
\newcommand{\dec}{\mathsf{dec}}

We quantify the speed of convergence of a sequence of iterates of \eqref{eqn:iterdef} in terms of its {\em $\delta$-decoupling time} $\dec_\delta(H_0)$, which is defined as the smallest $t$ at which some subdiagonal entry of $H_t$ satisfies $$|H_t(i+1,i)|\le \delta\|H_t\|,$$
where henceforth $\|\cdot\|$ will be used to denote the operator norm. In this context, ``rapid'' convergence means that $\dec_\delta(H_0)$ is a very slowly growing function of $n$ and $1/\delta$, ideally logarithmic or polylogarithmic. We will refer to a Hessenberg matrix with some $H(i+1,i)=0$ as {\em decoupled}.

\begin{remark}[Arithmetic Complexity from Decoupling Time] The motivation for the particular measure of convergence above is that there is a procedure called {\em deflation} which zeroes out the smallest subdiagonal entry of a $\delta$-decoupled Hessenberg matrix and obtains a nearby block upper triangular matrix, which allows one to pass to subproblems of smaller size incurring a backward error of $\delta\|H_0\|$. Repeating this procedure $n$ times (and exploiting the special structure of Hessenberg matrices to compute the $Q_t$ efficiently) yields an algorithm for computing a triangular $T$ and unitary $Q$ such that $\|H_0-Q^*TQ\|\le n\delta\|H_0\|$ in a total of  $O(n^3\dec_\delta(H_0))$ arithmetic operations \cite{watkins2007matrix}. Thus, the interesting regime is to take $\delta\ll 1/n$. We provide a full analysis of the deflation step as well as the total complexity of the algorithm in the subsequent paper \cite{banks2021global2}, but focus solely on decoupling in this paper. \end{remark}

In a celebrated work, Wilkinson \cite{wilkinson1968global} proved global convergence\footnote{i.e., from any initial condition $H_0$.} of shifted QR on all {\em real symmetric} tridiagonal\footnote{i.e., arising as the Hessenberg form of symmetric matrices.} matrices  using the shifting strategy that now carries his name. The linear convergence bound $\dec_\delta(H_0)\le O(\log (1/\delta))$ for this shifting strategy was then  obtained by Dekker and Traub \cite{dekker1971shifted} (in the more general setting of Hermitian matrices),  and reproven by Hoffman and Parlett \cite{hoffmann1978new} using different arguments. Other than these results for Hermitian matrices, there is no known bound on the worst-case decoupling time of shifted QR for any large class of matrices or any other shifting strategy. For nonnormal matrices, it is not even known if there is a shifting strategy which yields global convergence regardless of an effective bound on the decoupling time\footnote{A thorough discussion of related work appears in Section \ref{sec:related}.}. Shifted QR is nonetheless the most commonly used algorithm in practice for the nonsymmetric eigenproblem on dense matrices. The strategies implemented in standard software libraries heuristically converge very rapidly on ``typical'' inputs, but occasionally examples of nonconvergence are found \cite{day1996qr,moler2014} and dealt with in ad hoc ways.

Accordingly, the main theoretical
question concerning shifted QR, which has remained open since the 1960s, is:
\begin{qd}
    \label{qd}
    Is there a shifting strategy for which the Hessenberg shifted QR iteration provably and rapidly decouples on nonsymmetric matrices?
\end{qd}
\noindent Question \ref{qd} was asked in various forms e.g. by Parlett \cite{parlett1973normal,parlett1974rayleigh}, Moler \cite{moler1978three, moler2014}, Demmel \cite[Ch. 4]{demmel1997applied},  Higham \cite[IV.10]{higham2015princeton}, and Smale \cite{smale1997complexity} (who referred to it as a ``great challenge'').

 The main result of this article (Theorem \ref{thm:main}) is a positive answer to Question \ref{qd} which is quantified in terms of how nonnormal the input matrix is. To be precise, let $\H_\K^{n\times n}$ be the set of diagonalizable complex Hessenberg matrices $H_0$ with eigenvector condition number $\kappa_V(H_0)\le \K$. We exhibit a two parameter family of deterministic shifting strategies $\Sh_{k, \K}$ indexed by a degree parameter $k=2,4,8\ldots$ and condition parameter $B\ge 1$ and prove that:
\begin{enumerate}
    \item [(i)] The strategy $\Sh_{k, \K}$ satisfies $\dec_\delta(H_0)\le O(\log(1/\delta))$ for every $H_0\in\H_\K^{n\times n}$ and $\delta>0$.
    \item [(ii)] $\Sh_{k,\K}$ has degree $k$ and can be computed in roughly $O((\log k + \K^{\frac{\log k}{k}})k n^2)$ arithmetic operations, which is simply  $O(n^2k\log k)$ for the judicious setting $k=\Omega(\log B\log\log B)$. 
    \end{enumerate}Thus, the computational cost of the shifting strategy required for convergence increases as the eigenvectors of the input matrix become more and more ill-conditioned\footnote{As in many other settings, the condition number of the problem does not only have an effect on the numerical stability, but also on the  complexity of the algorithm. We refer the reader to Section 1.1.1 of \cite{banks2020pseudospectral} for a discussion on the condition number of the eigenvalue problem and its relation to $\kappa_V(H)$.}, but the dependence on the eigenvector condition number is very mild. 

We remark that such a result was not previously known even in the case $\K=1$, which corresponds to normal matrices. Further, as we explain in Remark \ref{remark:regularization}, a tiny random perturbation of any $H_0\in \H^{n\times n}$ is likely to be an element of $\H_\K^{n\times n}$ for small $\K$ (not depending on $H_0$). Thus, while our theorem does not give a single shifting strategy which works for all  matrices, it does give a strategy which  works for a tiny random perturbation of every matrix (with high probability, where ``tiny'' and ``small'' must be quantified appropriately). Consequently, the present work can be interpreted in the following three ways:

\begin{enumerate}
    \item [(A)] {\em An algorithm with guarantees on a restricted set of inputs.}  After fixing the parameter  $\K\geq 1$, our results provide a shifting strategy  which may be run on any input but is only guaranteed to succeed on those inputs $H_0$ satisfying $\kappa_V(H_0) \leq \K$. This may be of use when a priori information about the matrix is given. 
    \item [(B)] {\em An algorithm with guarantees for all inputs.}  One can  adopt a smoothed analysis perspective, in the sense of \cite{spielman2004smoothed}, and further endow the algorithm with a preprocessing step  consisting of randomly perturbing the input. This yields a guarantee (with high probability) of success for all inputs  and has the notable feature that the running time of the algorithm depends very mildly on the (inverse of the) size of the perturbation. However, in cases when the input matrix possesses some nice structure (e.g. sparsity), this a approach has a clear drawback, which points towards investigating the regularization effect of structure-preserving random perturbations. 
    
    \item [(C)] {\em A heuristic for why success occurs ``most of the time".}  One can interpret our results as a quantitative statement about the shifting strategy succeeding for ``almost all'' matrices. Then heuristically (and partially exiting the theoretical outlook of our work)  the rounding errors coming from floating point computations could potentially have a regularization effect on the input, ultimately implying convergence even on those inputs that did not satisfy the bound $\kappa_V(H_0)\leq \K$.

\end{enumerate}

\noindent {\em Motivation and Scope (Theory vs Practice).} Before proceeding to formally state the main results of this paper we would like to comment on the motivation and scope of our work. In this regard, it is important to distinguish theory from practice, and clarify that the present paper does not seek to be an immediate prescription for practitioners: more experimentation and theoretical development would be required to  perfect the existing (extremely sophisticated) libraries for diagonalization, which throughout the years have been fine-tuned to enhance their performance and efficiency, and constitute an engineering feat.

On the other hand, we do believe that our  work greatly advances the theory on the Hessemberg QR algorithm, whose convergence properties (as mentioned above) were mysterious even in the normal case. In particular, this paper provides (in exact arithmetic) a rigorous and conceptual understanding of: (i) the interaction between Ritz values based shifts and exceptional shifts (which was previously understood only in the case of unitary matrices \cite{wang2001convergence,wang2002convergence,wang2003convergence}, and has been exploited in practice in ad hoc ways), and (ii) the impact that ill-conditioned eigenvectors have on the speed of convergence  of the algorithm, together with the relevance of higher degree shifts (which are also used in practice) to nonnormality. The subsequent paper \cite{banks2021global2} further analyzes the numerical instabilities that arise near decoupling when shifted QR is implemented in finite precision arithmetic, as well as a full analysis of the deflation step, a topic that has been only partially addressed in the past \cite{parlett1993forward, watkins1995forward} even in the Hermitian setting.

Finally, we also believe that our work has the potential to eventually have a practical impact. As mentioned above, current implementations of the Hessenberg QR algorithm are a result of decades of improvement, where failure cases have been addressed by adding ad hoc shifts, ultimately leading to a complex algorithm which can still potentially fail to converge. For example,  \cite[pg. 3]{byers2007lapack} states ``As implemented in LAPACK 3.1, the new QR algorithm is substantially
more complicated... [and] has over 1,300 executable lines of code spread among seven subroutines." In contrast, the results here provide a conceptually simple shifting strategy which is infallible in the sense of (A) and (B), but which is unlikely to have comparable efficiency to the shifting strategies used in practice. This leads one to wonder if there is a shifting strategy that gets the best of both worlds. 
%In the meantime, the shifting strategies being used in practice can potentially be made more reliable simply by adding the strategy of this paper as a final case in the event that they fail to converge.

\subsection{Statement of Results and Key Ideas}\label{sec:results}
This paper contains two theorems. The first is a warmup to the main theorem, corresponding to the special case $k=2,B=1$, and states that a certain simple shifting strategy is always rapidly convergent on normal matrices in exact arithmetic, where we assume we can compute square roots exactly.
\begin{theorem}[Normal Matrices] \label{thm:mainnormal}
There is a degree $k=2$ deterministic shifting strategy $\Sh_{2,1}$ which ensures that 
    \begin{equation}\label{eqn:normaldectime}
        \dec_\delta(H_0)\le 4\log(1/\delta).
    \end{equation}
    for all normal Hessenberg $H_0$ and all $\delta>0$.
The worst case arithmetic complexity of each iteration of $\Sh_{2,1}$ is at most $$792\cdot T_\iqr(2,n)+O(1)$$
arithmetic operations, where $T_\iqr(2,n)\le 14n^2$ is the arithmetic complexity of an implicit QR step (see Section \ref{sec:prelims}).
\end{theorem}

\noindent {\em Key Ideas.} The proof of Theorem \ref{thm:mainnormal} appears in Section \ref{sec:normal}. The main challenge in proving such a theorem is that all known shifting strategies based on Ritz values have attractive fixed points which are not decoupled (e.g., certain unitary matrices), and the shifted QR iteration can stagnate (converge slowly) in the neighborhood of such matrices and possibly others. Our main conceptual idea is to carefully design a shifting strategy in such a way that {\em all of its stagnant states are well-characterized}.\footnote{The earliest precursor to such a result is Parlett \cite{parlett1966singular}, who showed that all of the { fixed points} (which can be thought of as perfectly stagnant states) of a certain degree $2$ shifting strategy are translations of unitary matrices.} Given our characterization, the shifting strategy then uses another mechanism (certain ``exceptional shifts'', see Section \ref{sec:related} for a discussion) to guarantee rapid convergence whenever stagnation occurs. At a technical level, our key innovation is to analyze convergence in terms of certain spectral measures associated with the iterates $H_t$, which enables our characterization. Our analytic approach  differs significantly from previous essentially algebraic  (e.g. \cite{wilkinson1968global, dekker1971shifted, hoffmann1978new, wang2001convergence}) and geometric (e.g. \cite{batterson1989dynamics, batterson1990convergence, batterson1990rayleigh, leite2013dynamics}) approaches to analyzing the shifted QR algorithm.

\begin{remark}[Optimizing the constants in Theorem \ref{thm:mainnormal}] 
\label{rem:optimizingmainnormal}
Our goal in presenting Theorem \ref{thm:mainnormal} is to elucidate one of the key mechanisms in our analysis by analyzing the simplest instantiation of our shifting strategy, prioritizing clarity of exposition over optimality. Later, in Lemma \ref{lem:exc} (which applies in a more general setting) we demonstrate that in the normal case the related strategy $\Sh_{4, 1}$ achieves \ref{eqn:normaldectime} using at most $54$ calls to a degree $4$ implicit QR step in each iteration, instead of the $792$ above. Furthermore, as discussed in Remark \ref{rem:optimizingmaintheoremgamma}, a more meticulous analysis in the normal case can reduce the constant even further, significantly decreasing it from $54$ to a smaller number.
\end{remark}

\begin{remark}[Maintaining Similarities] All arithmetic complexity bounds in this paper refer only to task of computing the Hessenberg iterates $H_t$, without keeping track of the accumulated unitary similarities $Q_0Q_1\ldots Q_t$ between $H_t$ and $H_0$. The additional task of computing the similarities is well-understood (e.g. see \cite{watkins2008qr}), and if desired can be achieved with a small increase in the total running time, without changing the asymptotic complexity. We focus on the convergence of the Hessenberg iterates $H_t$ and do not discuss maintaining the similarities further in this paper.
\end{remark}
Our second (and main) theorem generalizes Theorem \ref{thm:mainnormal} to nonnormal matrices. 
We need the following  notion to precisely state it. \\

\noindent {\em $\theta$-Optimal Ritz Values.} Like most previously studied shifting strategies, we define $\Sh_{k, \K}(H_t)$ in terms of  {\em Ritz values} of the current iterate $H_t$. In this paper, when working with a Hessenberg matrix, the term Ritz values will be exclusively used to refer to the eigenvalues of its bottom right $k\times k$ corner $\corner{(H)}{k}$;\footnote{In general, for a matrix $A$ and a tall matrix with orthonormal columns $Q\in \bC^{n\times k}$, the eigenvalues of $Q^*AQ$ are referred to as Ritz values. Here we are only interested in the case when $H$ is Hessenberg and $Q^*=[I_k, 0_{n-k} ]$.} they are also characterized variationally as the roots of the degree $k$ monic polynomial $p$ minimizing $\|e_n^*p(H)\|$, where $e_n$ is the last standard basis vector (see Lemma \ref{lem:minnorm} for details). Since computing eigenvalues of arbitrary matrices exactly is impossible, we assume access to a method for computing approximate Ritz values, in the sense encapsulated in the following definition.

\begin{definition}[$\r$-Optimal Ritz values and Ritz value finders] \label{def:approxritz}
Let $\theta\ge 1$. We call $\calR = \{r_1,\ldots,r_k\}\subset \C$ a set of $\r$-\emph{optimal Ritz values} of a Hessenberg matrix $H$ if
\begin{equation}
    \label{eqn:aritz}
    \left\|e_n^* \prod_{i\le k}(H-r_i)\right\|^{1/k}\le \r  \min_{p\in\calP_k} \|e_n^*p(H)\|^{1/k}.
\end{equation}
A {\em Ritz value finder} is an algorithm
$\apr(H, k, \theta )$ that takes as inputs a Hessenberg matrix $H\in \bC^{n\times n}$, a positive integer $k$ and an accuracy parameter $\theta>1$, and outputs a set $\calR = \{r_1, \dots, r_k\}$ of $\r$-optimal Ritz values of $H$ whenever the right hand side of $\eqref{eqn:aritz}$ is nonzero. Let $\tapr(k, \theta, \delta)$ be the maximum number of arithmetic operations used by $\apr(H,k,\theta)$ over all inputs $H$ such that the right hand side of $\eqref{eqn:aritz}$  satisfies\footnote{Such a lower bound is needed, since otherwise we could use $\apr$ to compute the eigenvalues of $\corner{H}{k}$ to arbitrary accuracy in finite time.} $$\min_{p\in\calP_k} \|e_n^*p(H)\|^{1/k}\ge \delta\|H\|.$$
\end{definition}

A Ritz value finder satisfying Definition \ref{def:approxritz} can be efficiently instantiated using polynomial root finders (e.g. \cite{pan2002univariate}) or other provable eigenvalue computation algorithms (e.g. \cite{banks2020pseudospectral, banks2021global3}) with guarantees of type $\tapr(\theta,k,\delta)=O(k^c\log(\tfrac{1}{\delta(\theta - 1)}))$. We defer a detailed discussion of numerical issues surrounding this implementation to our companion papers \cite{banks2021global2,banks2021global3}. The subtlety of not being able to compute Ritz values exactly is secondary to the dynamical phenomena which are the focus of this paper, so on first reading of the proofs it is recommended to assume $\theta=1$ (i.e., Ritz values are computed exactly), even though this is unrealistic when $k > 1$. The theorem below is stated with $\r = 2$, which is also the parameter setting used in \cite{banks2021global2}.

We now present our main theorem. All logarithms are base $2$. 
\begin{theorem}[Nonnormal Matrices]
\label{thm:main} There is a family of deterministic shifting strategies  $\Sh_{k, \K}$ (described in Section \ref{sec:mainresult}) parameterized by degree $k=2,4,8,\ldots$  and nonnormality bound $B\ge 1$ with the following properties.
\begin{enumerate}
    \item (Rapid Decoupling) If $H_0\in\H^{n\times n}_B$ , then for every $\delta>0$, the QR iteration with strategy $\Sh_{k, \K}$ satisfies 
    \begin{equation}\label{eqn:maindeflate}
        \dec_\delta(H_0)\le 4\log(1/\delta).
    \end{equation}
    \item (Cost Per Iteration Before Decoupling) Given a Ritz value finder $\apr(H,k,\theta)$ with complexity $\tapr(k,\theta, \delta)$, an accuracy parameter $\delta>0$, and a Hessenberg matrix $H_t\in \H^{n\times n}_B$, computing $H_{t+1}$ given $H_t$ has a cost per iteration of at most 
    \begin{equation}
        \label{eq:main-arith-ops}
         \left( \log k  +   \tnet\left(0.002\, \K^{-\frac{8\log k + 4}{k-1}}\right)\right)\cdot T_{\iqr}(k,n) + \tapr(k, 2, \delta) + \log k
    \end{equation}
    arithmetic operations for all iterations before \eqref{eqn:maindeflate} is satisfied, where $\tnet(\epsilon)\le 4/\epsilon^{2}$ is number of points in an efficiently computable $\epsilon$-net of the unit disk and $T_{\iqr}(k,n)\le 7kn^2$ is an upper bound on the arithmetic cost of a degree $k$ implicit QR step (see Section \ref{sec:prelims}). 
    % where
    % \begin{equation}\label{eqn:nn}
    %   \NN(k,\K):= 12kn^2\cdot 131 \K^{\frac{8\log k + 8}{k - 1}} .
    % \end{equation}
\end{enumerate}
\end{theorem}

The term involving $\tnet$ captures the the cost of performing certain ``exceptional shifts'' (see Section \ref{sec:related}) used in the strategy. The tradeoff between the nonnormality of the input matrix and the efficiency of the shifting strategy appears in the cost of the exceptional shift, where it is seen that setting
\begin{equation}\label{eqn:kappasetting}
    k=\Omega(\log \K\log\log \K)
\end{equation}
yields $\K^{-\frac{8\log k + 4}{k-1}}=\Omega(1)$ and a consequent total running time of $O(n^2k\log k)$ operations per iteration. Note that the bound  $\K\ge \kappa_V(H_0)$ must be known in advance in order to determine how large a $k$ is needed to make the cost of the exceptional shift small. One may also take $k$ to be a constant independent of $\K$, but this causes the arithmetic complexity of each iteration to depend polynomially on $\K$ rather than logarithmically. Note that for normal matrices one may take $k=2$ and $\K=1$.\\

\begin{remark}[Higher Degree Shifts]\label{rem:highershifts} In exact arithmetic, a QR step with a degree $k$ shift $p(z)=(z-r_1)\ldots(z-r_k)$ is identical to a sequence of $k$ steps with degree $1$ shifts $(z-r_1),(z-r_2),\ldots,(z-r_k)$  (see e.g. \cite{watkins2007matrix} for a proof), so any degree $k$ strategy can be simulated by a degree $1$ strategy while increasing the iteration count by a factor of $k$.\footnote{This also has some important  advantages with regards to numerical stability, which are discussed in \cite{banks2021global2}.} We choose to present our strategy as higher degree for conceptual clarity. The overall performance  of shifting strategies of degrees as high as $k=180$ has been tested in the past \cite[Section 3]{braman2002multishift} and  $k=50$ is often used in practice \cite{KressnerQuote}. 
\end{remark}

\noindent {\em Key Ideas.} The proof of Theorem \ref{thm:main} appears in Section \ref{sec:mainresult}. The main new difficulty in the nonnormal case is that  the iterates $H_t$ can behave chaotically on short time scales,\footnote{ We measure time not as the number of QR steps, but as the number of QR steps of degree $1$, so for example a QR step with a degree $k$ shift corresponds to $k$ time steps.} lacking any kind of obvious  algebraic or geometric monotonicity properties (which are present in the normal case, see e.g. \cite{parlett1974rayleigh,batterson1994convergence}). This lack of monotone quantities makes it hard to reason about convergence, as noted by Parlett \cite{parlett1974rayleigh}. For example, consider the family of $n\times n$ matrices: 
\begin{align}\label{eqn:exampleB}
    M = \begin{pmatrix} 
         & & & &  \beta_n \\
        \beta_1 & & & &  \\
        & \beta_2 & & &  \\
        & &  \ddots &  & \\
        & & &  \beta_{n-1} &
    \end{pmatrix}
    \end{align} 
where $\beta_1,\ldots,\beta_n\in (0,1)$. Observe that, for $k\leq n-1$, the characteristic polynomial of the bottom right corner $\corner{M}{k}$ is just $z^k$, so any na\"ive shifting strategy based on Ritz values will yield the trivial shift. One can verify that a QR step with the trivial strategy applied to $M$  cyclically permutes the $\beta_i$, while leaving the zero pattern of $M$ intact. This means that for adversarially chosen $\beta_1,...,\beta_n$, the bottom few subdiagonal entries of $M$ --- the traditional place to look for monotonicity in order to prove convergence (see e.g. \cite{hoffmann1978new}) --- exhibit arbitrary behavior over a small number of $QR$ steps. At very long time scales of $n$ steps, the behavior becomes periodic and predictable, but there is still no convergence. 

We surmount the above difficulty by using {\em higher degree shifts}. The key insight is that if there is an upperbound on $\kappa_V(H_t)$, then the behavior of a degree $k=\log\kappa_V(H_t)$ QR step is quite predictable, and can be analyzed similarly to the normal case --- essentially, this corresponds to $k$ degree $1$ steps which is enough to ``damp'' the transient behavior due to nonnormality (this is articulated precisely in Section \ref{sec:mainresult}). To see this phenomenon in action, if we impose a bound on $\kappa_V(M)$ in Example \eqref{eqn:exampleB}, it can be seen that the ratios of the $\beta_i$ cannot be arbitrary and the geometric mean of the bottom $k=\log\kappa_V(M)$ subdiagonal entries of $M$ must remain almost-constant on time intervals of $k$ unshifted $QR$ steps. 

At a technical level, the main obstacle is that unlike normal matrices, nonnormal matrices do not have spectral measures or a corresponding continuous functional calculus (which was used crucially in our characterization of stagnation in the proof of Theorem \ref{thm:mainnormal}). A key ingredient in the proof of Theorem \ref{thm:main} is a notion of ``approximate functional calculus'' which allows us to recover enough analytic structure in the nonnormal case to execute the same proof strategy as in the normal case.

\begin{remark}[Regularization of $\kappa_V$ by Random Perturbation]
\label{remark:regularization}
The \emph{pseudospectral regularization} guarantees from \cite{banks2021gaussian} (verifying the conjecture of \cite{davies2008approximate}) imply that every matrix $A$ is $\acc\|A\|$-close in operator norm to matrices with $\kappa_V=O(n^2/\acc)$. Such a nearby, well-conditioned matrix can be produced (with high probability) by perturbing each entry of $A$ with an independent complex Gaussian of variance $\delta$\footnote{See also \cite{banks2020overlaps, jain2020real} for more general perturbations, including real perturbations, which have a similar effect and \cite{armentano2018stable} for similar results in the complex Gaussian case. }. After perturbing we can thus (with high probability) take $\K=O(n^2/\acc)$ in Theorem \ref{thm:main} and set  $k = O(\log(n/\acc)\log\log(n/\acc))$, incurring a backward error of $\delta$ and yielding a per iteration arithmetic cost of $O(n^2\log(n/\delta)\log\log^2(n/\delta))$ before $\delta$-decoupling. This approach embraces the fact that the shifted QR algorithm can only in the first place guarantee backward accuracy of the eigenvalues it computes, so there is no harm in using an initial small random perturbation as a ``preconditioning'' step.  \end{remark}

\begin{remark}[Numerical Stability,  Deflation, and Bit Complexity]
The shifting strategy in Theorem \ref{thm:main} can be implemented in floating point arithmetic using $O(k\log(n/\delta))$ bits of precision for the implicit QR steps\footnote{Hence, when a random perturbation is used as a preconditioner,  in view of Remark \ref{remark:regularization} the number of bits of precision required becomes $O( \log^2(n/\delta)\log\log(n/\delta))$.} and $O(k^2\log^2(n/\delta))$ bits of precision for the Ritz value finder\footnote{The Ritz value finder is invoked only on small $k\times k$ matrices and incurs a subdominant cost.}, while preserving both correctness and rapid convergence, with the caveat that the numerical implementation requires using randomization in order to be efficient. 
This is proved in the companion papers \cite{banks2021global2, banks2021global3}, along with a detailed analysis of deflation, yielding a complete algorithm for computing the eigenvalues of a matrix with good bit complexity estimates. 
\end{remark}

When succeeding, the QR iteration computes the Schur factorization of the input matrix $A$, which consists of a triangular matrix (whose diagonal entries are the eigenvalues of $A$), and a unitary matrix (from where the eigenvectors of $A$ can be recovered). The computation of the unitary part is optional, and requires keeping track of the unitary conjugations used throughout the iteration, a task that incurs a higher cost in the running time (e.g. see \cite{watkins2008qr}). In this paper (and the subsequent works \cite{banks2021global2, banks2021global3}) we will focus on the computation of the triangular part (which is the part of the algorithm that was not yet understood), and the running times appearing in the results will refer to exclusively this task. We note however, that a running time for obtaining the full Schur form (i.e. computing the unitary part too), can be obtained directly from our results and will have the same asymptotic running time, just with larger constants.

\subsection{History and Related Work}\label{sec:related}
The literature on shifted QR is vast, so we  mention only the most relevant works --- in particular, we omit the large body of experimental work and do not discuss the many works on local convergence of shifted QR  (i.e., starting from an $H_0$ which is already very close to decoupling). The reader is directed to the excellent surveys \cite{batterson1995dynamical, smale1997complexity, chu2008linear} or \cite{parlett2000qr, watkins2008qr, golub2009qr} for a dynamical or numerical viewpoint, respectively, or to the books \cite{golub1996matrix,trefethen1997numerical,demmel1997applied,watkins2007matrix} for a comprehensive treatment. A detailed historical summary appears in \cite{golub2009qr}.

Most of the shifting strategies studied in the literature are a combination of the following three types. The motivation for considering shifts depending on $\corner{H}{k}$ is closely related to Krylov subspace methods, see e.g. \cite{watkins2007matrix}. Below $H$ denotes the current Hessenberg iterate.
\begin{enumerate}
    \item \textit{$k$-Francis Shift.} Take $p(z)=\det(z-\corner{H}{k})$ for some $k$. The case $k=1$ is called Rayleigh shift.
    \item \textit{Wilkinson Shift.} Take $p(z) =(z-a)$ where $a$ is the root of $\det(z-\corner{H}{2})$ closer to $\corner{H}{1}$.
    \item \textit{Exceptional Shift.} Let $p(z)=(z-x)$ for some $x$ chosen randomly or arbitrarily, perhaps with a specified magnitude (e.g. $|x|=1$ for unitary matrices in \cite{eberlein1975global,wang2001convergence,wang2002convergence,wang2003convergence}).
\end{enumerate}
Shifting strategies which combine more than one of these through some kind of case analysis are called ``mixed'' strategies. \\

\noindent {\em Symmetric Matrices.} Jiang \cite{erxiong1992note} showed that the geometric mean of the bottom $k$ subdiagonal entries is monotone for the $k$-Francis strategy in the case of symmetric tridiagonal matrices. Aishima et al. \cite{aishima2012wilkinson} showed that this monotonicity continues to hold for a ``Wilkinson-like'' shift which chooses $k-1$ out of $k$ Ritz values. Both of these results yield global convergence on symmetric tridiagonal matrices (without an effective bound on the number of iterations).\\

\noindent{\em Rayleigh Quotient Iteration and Normal Matrices.} The behavior of shifted QR is well known to be related to shifted inverse iteration (see e.g. \cite{trefethen1997numerical}). In particular, the Rayleigh shifting strategy corresponds to a vector iteration process known as Rayleigh Quotient Iteration (RQI). Parlett \cite{parlett1974rayleigh} (building on \cite{ostrowski1957convergence, buurema1958geometric, parlett1968convergence}) showed that RQI converges globally (but without an effective bound) on almost every normal matrix and investigated how to generalize this to the nonnormal case.

Batterson \cite{batterson1990convergence} studied the convergence of $2$-Francis shifted QR on $3\times 3$ normal matrices with a certain exceptional shift and showed that it always converges. The subsequent work \cite{batterson1994convergence} showed that $2$-Francis shifted QR
converges globally on almost every real $n\times n$ normal matrix (without an effective bound).  In Theorem 6 of that paper, it was shown that the same potential that
we consider is monotone-decreasing when the $k$-Francis shift is run on normal
matrices, which was an inspiration for our proof of almost-monotonicty for
nonnormal matrices.\\

\noindent {\em Nonnormal Matrices.} Parlett \cite{parlett1966singular} showed that an unshifted QR step applied to a singular matrix leads to immediate \mbox{$0$-decoupling}, taking care of the singularity issue that was glossed over in the introduction, and further proved that  all of the fixed points of an extension of the $2$-Francis shifted QR step (for general matrices) are multiples of unitary matrices. 

In a sequence of works, Batterson and coauthors investigated the behavior of RQI and $2$-Francis on nonnormal matrices from a dynamical systems perspective. Batterson and Smillie \cite{batterson1989dynamics,batterson1990rayleigh} showed that there are real matrices such that RQI fails to converge for an open set of real starting vectors. The latter paper also established that RQI exhibits chaotic behavior on some instances, in the sense of having periodic points of infinitely many periods.  Batterson and Day \cite{batterson1992linear} showed that $2$-Francis shifted QR converges globally and linearly on a certain conjugacy class of $4\times 4$ Hessenberg matrices.

In the realm of periodicity and symmetry breaking, Day \cite{day1996qr}, building on an example of Demmel, showed that there is an open set of $4\times 4$ matrices on which certain mixed shifting strategies used in the EISPACK library  fail to converge rapidly in exact arithmetic; such an example was independently discovered by Moler \cite{moler2014} who described its behavior in finite precision arithmetic. These examples are almost normal in the sense that they satisfy $\kappa_V\le 2$, so the reason for nonconvergence is symmetry, and our strategy $\Sh_{k,B}$ with modest parameters $k=B=2$ is guaranteed to converge rapidly on them (in exact arithmetic).

Using topological considerations, Leite et al. \cite{leite2013dynamics} proved that no continuous shifting strategy can decouple on every symmetric matrix. Accordingly (in retrospect), the most successful shifting strategy for symmetric matrices, the Wilkinson Shift, is discontinuous in the entries of the matrix and explicitly breaks symmetry when it occurs. Our strategy $\Sh_{k, \K}$ is also discontinuous in the entries of the matrix.\\

\noindent {\em Mixed and Exceptional Shifts.} Eberlein and Huang \cite{eberlein1975global} showed global convergence (without an effective bound) of a certain mixed strategy for unitary Hessenberg matrices; more recently, the works \cite{wang2001convergence,wang2002convergence,wang2003convergence} exhibited mixed strategies which converge globally and linearly for unitary Hessenberg matrices with a bound on the rate\footnote{We refer to the constant multiplying the error in each iteration of linear convergence as its {\em rate}.}, but this bound depends on the matrix in a complicated way and is not clearly bounded away from $1$.  Our strategy $\Sh_{k, \K}$ is also a mixed strategy which in a sense combines all three types above. Our choice of exceptional shift was in particular inspired by the work of \cite{eberlein1975global,wang2002convergence} --- the difference is that the size of the exceptional shift is naturally of order $1$ in the unitary case, but in the general case it must be chosen carefully at the correct spectral scale. \\

\noindent {\em Higher Degree Shifts.} The idea of using higher degree shifts was already present in \cite{francis1961qr, dekker1971shifted}, but was popularized in by Bai and Demmel in \cite{bai1989block}, who observed that higher order shifts can sometimes be implemented more efficiently than a sequence of lower order ones; see \cite[Section 3]{bai1989block} for a discussion of various higher order shifting strategies which were considered in the 1980s. More modern approaches use multishifts \cite{braman2002multishift} in combination with other techniques such as aggressive early deflation \cite{braman2002multishiftII}. 

The use of higher degree shifts in this paper is purely in order to tame the effects of nonnormality, and {\em not} to improve efficiency as in the aforementioned previous works. As explained in  Remark \ref{rem:highershifts}, our algorithms can be equivalently described using sequences of single shifts. \\

%\noindent {\em Integrable Systems.} The unshifted QR algorithm on Hermitian  matrices is known to correspond to evaluations of an integrable dynamical system called the Toda flow at integer times \cite{deift1983ordinary}; such a correspondence is not known for any nontrivial shifting scheme or for nonnormal matrices.  See \cite{chu2008linear} for a detailed survey of this connection. More recently, the line of work \cite{pfrang2013long,deift2014universality,deift2019universality} studied the universality properties of the decoupling time of unshifted QR on random matrices, and used the connection to Toda flow to prove universality in the symmetric case; it was experimentally observed that such universality continues to hold for shifted QR.

We defer a detailed discussion of the extensive related work on numerical issues related to shifted QR as well as a comparison to other algorithms for computing eigenvalues (in particular, \cite{armentano2018stable} and \cite{banks2020pseudospectral}) to our companion paper \cite{banks2021global2}.\\

\section{Preliminaries and Notation}
\label{sec:prelims}
As mentioned above, the  eigenvector condition number of a  diagonalizable matrix $M$ is defined as 
$$ 
    \kappa_V(M) := \inf_{V:M=VDV^{-1}} \|V\|\|V^{-1}\|.
$$ 
We note that, since $\|V\|\|V^{-1}\|$ is invariant under multiplying $V$ by a scalar, the infimum is not changed when restricting to the compact set of eigenvector matrices with $\|V\|\leq 1$, and therefore the infimum is always attained by some matrix $V_0$. Moreover, by van der Sluis theorem (see \cite[Theorem 7.5]{higham2002accuracy}), the condition number of $V_0$ differs at most by a factor of $\sqrt{n}$ from the eigenvector matrix with unit columns.  

Throughout the remainder of the paper, $H = (h_{i,j})_{i,j \in [n]}$ will denote an $n\times n$ upper Hessenberg matrix. As in the introduction, we use
$$\corner{H}{k}\quad \text{and} \quad \chi_k(z)$$
to denote the lower-right $k\times k$ corner of $H$ and its characteristic polynomial respectively. Following the convention in operator theory, we will write scalar multiples of the identity $zI$ with $z\in \C$ as $z$, e.g., we will write $z-H$ instead of $zI-H$. As mentioned in the introduction, $\calP_k$ will always denote the set of monic polynomials of degree $k$.  All matrix norms are operator norms, denoted by $\|\cdot\|$.\\

\noindent {\em Probabililty.} We use standard probabilistic notation: $\P[\cdot]$ denotes the probability of an event and $\E[\cdot]$ denotes the expectation of a random variable. All probabilities and expectations in this paper are with respect to the random variable ``$Z_H$'', defined by \eqref{eqn:normrand} in Section \ref{sec:normal} and more generally by \eqref{eqn:funccalcrand} in Section \ref{sec:mainresult}; the random variable in any probabilistic statement will therefore always be unambiguous from the context. The only probabilistic facts we use are linearity of expectation, Jensen's inequality, and the Paley-Zygmund inequality.
\\

\noindent {\em Implicit QR.} We assume black box access to a routine (the implicit QR algorithm) for efficiently performing a QR step in $O(kn^2)$ arithmetic operations (rather than the $O(kn^3)$ operations required by a naive method). Since numerical stability issues are not discussed in this paper, this is the only property that we will use of the implicit QR algorithm.
\begin{definition}[Implicit QR Algorithm]
\label{def:stableiqr}
For $k\leq n$, an exact implicit QR algorithm $\exactqr(H, p(z))$ takes as inputs an irreducible\footnote{A Hessenberg matrix $H$ is said to be irreducible if all of its subdiagonal entries are non-zero.} Hessenberg matrix $H \in \bC^{n\times n}$ and a polynomial $p(z)=(z-s_1)\cdots (z-s_k)$ and outputs a Hessenberg matrix $\next H$ satisfying 
$$\next H = Q^* H Q,$$
where $Q$ is a unitary matrix such that $p(H) = QR$ for some upper triangular matrix $R$, as well as the number $\|e_n^\ast p^{-1}(H)\|$ (which appears as the bottom right entry of $R$) whenever $p(H)$ is invertible.  It runs in at most 
\begin{equation}
    T_{\iqr}(k,n)\le 7kn^2
\end{equation} operations.
\end{definition}

Note that in the above definition we have assumed that the input to $\exactqr$ is an \emph{irreducible} Hessenberg matrix. This will not be a problem throughout the paper in the analysis of the algorithm, since the ultimate goal is to prove an upper bound on the number arithmetic operations needed to achieve $\delta$-decoupling, and by definition \emph{reducible} matrices are $\delta$-decoupled for any $\delta>0$.   We refer the reader to  \cite[Section 3]{watkins2008qr} for a proof in exact arithmetic of the existence of an efficient implicit QR algorithm.\\

\noindent {\em Potential $\psi_k$.} We will use the geometric mean of the last $k$ subdiagonal entries of the $H$ to track convergence of the Shifted QR iteration, since we are guaranteed $\delta$-decoupling once this quantity is smaller than $\delta\|H\|$. More explicitly, we introduce the following definition. 

\begin{definition}[Potential function $\psi_k(H)$]
    The \textit{potential} \footnote{We remark in passing that $\log\psi_k(H)$ is an upper bound on a certain logarithmic potential associated with the Ritz values in the sense of potential theory, though we do not use this relation in this paper (the interested reader can consult \cite{saff2013logarithmic}). The term ``potential" here is alluding to the notion of ``potential function" often used in computer science as a tool to track the progress of an algorithm. } 
$\psi_k(H)$ of $H$ to be
\begin{equation}\label{eqn:potdef}
    \pot(H) := |h_{n-k, n-k-1}\cdots h_{n, n-1}|^{\frac{1}{k}}.
\end{equation}
\end{definition}

\label{sec:basiclemmas}
We record two useful lemmas relating the potential $\pot(H)$, the Hessenberg structure of $H$, and their evolution under the shifted QR iteration. The first gives a variational characterization of the potential (see \cite[Theorem 34.1]{trefethen1997numerical}). 

\begin{lemma}[Variational Formula for $\pot$]
    \label{lem:minnorm}
    Let $H\in \bC^{n\times n}$ be any Hessenberg matrix. Then, for any $k$ 
    $$
        \psi_k(H) = \min_{p\in \calP_k}\|e_n^* p(H)\|^{\frac{1}{k}},
    $$
    with the minimum attained for $p= \chi_{k}$. 
\end{lemma}

\begin{proof}
    Since $H$ is upper Hessenberg, for any polynomial $p\in \calP_k$ we have
    $$
        p(H)_{n, n-j} = \begin{cases} p(\corner{H}{k})_{k, k-j+1} & j=0, \dots, k-1, \\
        h_{n-k,n-k-1} \cdots h_{n,n-1} & j = k, \\
        0 & j \ge k+1.
        \end{cases}    
    $$
    Thus for every such $p$,
    $$
        \min_{p\in \calP_k}\|e_n^* p(H)\|\geq |h_{n-k, n-k-1}\cdots h_{n, n-1}| = \psi_k(H)^k,
    $$
    and the bound will be tight for any polynomial whose application to $\corner{H}{k}$ zeroes out the last row; by Cayley-Hamilton, the matrix $\chi_k(\corner{H}{k})$ is identically zero.
\end{proof}

The second lemma gives a mechanism for proving upper bounds on the potential of $\next H=\exactqr(H,p(z))$ in terms of the shift polyomial $p$. Consequently, the following quantity will prove useful. 

\begin{definition}[$\tau_p(H)$]
    For a monic polynomial $p \in \calP_k$  define
\begin{equation}
    \label{eq:tau-def}
    \tau_p(H) := \|e_n^\ast p(H)^{-1}\|^{-\frac{1}{k}},
\end{equation}
when $p(H)$ is invertible, and $\tau_p(H)=0$ otherwise.
\end{definition}

 The special case $k=1$ of the $\tau_p(H)$ quantity has been used to great effect in previous work studying linear shifts (e.g. \cite{hoffmann1978new}), and our next lemma shows that it bounds the potential of $\next H = \exactqr(H,p(z))$ for shift polynomials $p$ of arbitrary degree. 

\begin{lemma}
    [Upper Bounds on $\pot(\next H)$]
    \label{lem:lowerboundontau}
    Let $H\in \bC^{n\times n}$ be a Hessenberg matrix, $p(z)$ a monic polynomial of degree $k$ and $\next H = \exactqr(H,p(z))$. Then 
        $$ 
        \pot(\next  H)\leq \tau_p(H).$$
\end{lemma}

\begin{proof}
    Assume first that $p(H)$ is singular. In this case for any QR decomposition $p(H) = QR$, the entry $R_{n,n} = 0$, and because $p(\next{H}) = Q^* p(H) Q = RQ$, the last row of $p(\next H)$ is zero as well. In particular $\pot(\next H) = |p(\next H)_{1, k+1}|^{\frac{1}{k}} =0 = \tau_p(H)$. When $p(H)$ is invertible, applying Lemma \ref{lem:minnorm} and using repeatedly that $Q$ is unitary, $R$ is triangular, and $p(H) = QR$,
    \begin{align*}
        \pot(\next H)^k
        \le \|e_n^\ast p(\next H)\|
        =\|e_n^\ast Q^\ast p(H)\|
        = \|e_n^\ast R\|
        % = \|e_n^\ast R^{-1}\|^{-1}
        = \|e_n^\ast R^{-1}Q^\ast\|^{-1}
        = \|e_n^\ast p(H)^{-1}\|^{-1}
        = \tau_p(H)^k.
    \end{align*}
\end{proof}

Lemma \ref{lem:lowerboundontau} ensures that given $H$, we can reduce the potential with an implicit QR step by producing a polynomial $p$ with $\tau_p(H)=\|e_n^\ast p(H)^{-1}\|^{-\frac{1}{k}} \le (1-\gamma)\pot(H)$. Note that if the roots of $p(z)$ are close to the eigenvalues of $H$, then we expect $\|e_n^*p(H)^{-1}\|$ to be large, and therefore $\tau_p(H)$ to be small, which articulates that shifts that are close to the eigenvalues accelerate convergence.

\section{Normal Matrices}
\label{sec:normal}
In this short section we describe the simplest shifting strategy in our family, $\Sh_{2,1}$, and prove that it enjoys global linear convergence with a {uniform} rate\footnote{i.e., the error is multiplied by a fixed constant in each iteration.} for {all} normal matrices, improving the qualitative results of \cite{parlett1968convergence}. The strategy is inspired by the Wilkinson shift in that it chooses the shift to be one of the roots of $\chi_2$ the characteristic polynomial of $\corner{H}{2}$, but it does so in a ''greedy'' manner which is based on the quantity $\tau$ defined in \eqref{eq:tau-def}. If one of these shifts doesn't make adequate progress towards convergence, the strategy resorts to certain carefully chosen exceptional shifts, one of which is guaranteed to do so. We track convergence of the strategy using the potential $\psi_2(H)$ defined in \eqref{eqn:potdef}, which is simply the geometric mean of the bottom two subdiagonal entries of $H$.\footnote{The  exact same potential function was used by Hoffman and Parlett \cite{hoffmann1978new} to analyze the Wilkinson shift on symmetric matrices. } 

We will heavily use the functional calculus for normal matrices (a standard tool from functional analysis \cite[Chapter VII]{reed1980functional}) to describe and analyze our strategy. Recall that if $H$ has eigenvalues $\Lambda(H)=\{\lambda_1,\ldots,\lambda_n\}\subset \C$ and corresponding orthonormal eigenvectors $v_1,\ldots,v_n\in \C^n$, then for any analytic function $f$ defined in a neighborhood of $\Lambda$:
\begin{equation} \label{eqn:normcalc} \int |f(z)|^2 d\mu(z) = \|e_n^* f(H)\|_2^2,\end{equation}
where 
\begin{equation}\label{eqn:specmeasurenormal}
    \mu :=\sum_{i\le n} |e_n^*v_i|^2 \delta_{\lambda_i}
\end{equation}
is the spectral (probability) measure corresponding to $e_n$. We  use probabilistic instead of linear algebraic notation to make our proofs more transparent: letting $Z_H$ denote the random variable taking values in $\{\lambda_1,\ldots,\lambda_n\}$ with distribution $\mu$ as above, \eqref{eqn:normcalc} can be rewritten as
\begin{equation}\label{eqn:normrand}
    \E [|f(Z_H)|^2] = \|e_n^*f(H)\|^2.
\end{equation}

\begin{remark}[Probabilistic Notation]
We stress that our shifting strategy is deterministic and we use random variables for notational convenience only. The proofs could be equivalently written in terms of inner products involving the eigenvectors of $H$, but this would obscure the convexity and probabilistic inequalities that drive them.
\end{remark}
 The shifting strategy is presented below.  Recall that an $\epsilon-$net of a compact set $X\subset \C$ is a finite set of points $\calN\subset X$ such that 
$$\max_{x\in X}\min_{y\in \calN} |x-y| \le \epsilon.$$ Note that the equality of the last two expressions in Line 1 relies on \eqref{eqn:normrand}.\\

\noindent 
\begin{boxedminipage}{\textwidth}
$$\Sh_{2,1}, \text{``Greedy Wilkinson Shift for Normal Matrices''}$$
    \textbf{Input:} Hessenberg $H$, decoupling rate $\gamma\in (0,1)$ \\
    \textbf{Output:} Hessenberg $\next{H}$\\
    \textbf{Requires:} $0 < \psi_2(H)$\\
    \textbf{Ensures:} $\psi_2(\next{H})\le (1-\gamma)\psi_2(H)$.\\
    \begin{enumerate}
        \item \label{line:shn1} Let $p(z)=(z-r_1)(z-r_2)=\det(z-\corner{H}{2})$. Choose $$r=\arg\max_{i=1,2} \|e_n^*(H-r_i)^{-2}\|=\arg\max_{i=1,2} \E [|Z_H-r_i|^{-4}]^{1/2},$$ computing these quantities using $\exactqr(H,(z-r_i)^2)$ for $i=1,2$.
        \item \label{line:shn2} If $\psi_2(\exactqr(H,(z-r)^2)) \le (1-\gamma) \psi_2(H)$, output $\next H = \exactqr(H,(z-r)^2)$
        \item \label{line:shn3} Else, let $\calS$ be a $\epsilon\psi_2(H)-$net of $D(r,\sqrt{3}\psi_2(H))$ with $\epsilon=(1-\gamma)^2/\sqrt{27}.$      
        \item \label{line:shn4} For each $s \in \calS$, if $\psi_2(\exactqr(H,(z - s)^2)) \le (1-\gamma)\psi_2(H)$, output $\next H = \exactqr(H,(z-s)^2)$
    \end{enumerate}
\end{boxedminipage}\\ \\

The key observation is that if  $\Sh_{2,1}$ fails to significantly decrease $\psi_2$ in Line 2, then the spectral measure of $H$ with respect to $e_n$ must be significantly supported on a disk of radius roughly $\psi_2(H)$ centered at $r$.
\begin{lemma}[Stagnation Implies Support, Normal Case]
    \label{lem:mainnormal}
    Let $\gamma\in (0,1)$ and let $r$ be the Greedy Wilkinson shift for an upper Hessenberg matrix $H$. If
\begin{equation}
    \label{eqn:jstagnormal} \psi_2\left(\exactqr(H,(z - r)^2)\right)\geq (1-\gamma)\psi_2(H)>0
\end{equation}
then for every $t\in (0,1)$:
    \begin{equation}
        \label{eqn:pz}
          \P \left[ |Z_H-r|\le \frac{1}{\sqrt{t}}\psi_2(H)\right]
        \ge (1-t)^2(1-\gamma)^4
    \end{equation}
\end{lemma}

\begin{proof} Observe that $H-r$ is invertible since otherwise, for $\next{H}= \exactqr(H,(z - r)^2)$, we would have $\psi_2(\next{H})=0$ by Lemma \ref{lem:lowerboundontau}.
   Our assumption implies that:
    \begin{align*}
        (1-\gamma)\psi_2(H) &\le \psi_2(\next{H}) & &\text{hypothesis}\\
        &\le \tau_{(z - r)^2}(H) & &\text{Lemma \ref{lem:lowerboundontau}}\\
        &= \|e_n^\ast(H-r)^{-2}\|^{-1/2} & &\text{definition}\\
        &= \left({\E\left[\frac{1}{|Z_H - r|^{4}}\right]}\right)^{-1/4} & &\text{by \eqref{eqn:normrand}}\\
        &\le\left(\E\left[\frac{1}{|Z_H - r|^{2}}\right]\right)^{-1/2} & &\text{Jensen, $x\mapsto x^2$}\\
        &\le \left(\frac{1}{2}\sum_{i=1}^2\E\left[\frac{1}{|Z_H - r_i|^{2}}\right]\right)^{-1/2} & &\text{choice of $r$ in Line $1$ and \eqref{eqn:normrand}}\\
        &=\left(\E\left[\frac{1}{2}\sum_{i=1}^2\frac{1}{|Z_H - r_i|^{2}}\right]\right)^{-1/2} & &\text{Fubini}\\
        &\le\left(\E\left[\frac{1}{|Z_H - r_1||Z_H-r_2|}\right]\right)^{-1/2} & &\text{AM/GM}\\       &=\left(\E\left[\frac{1}{|\chi_2(Z_H)|}\right]\right)^{-1/2} & &\text{definition of $\chi_2$}\\ 
        &\le \left(\E \left[|\chi_2(Z_H)|\right]\right)^{1/2}& &\text{Jensen, $x\mapsto 1/x,x>0$}\\
        &\le \left(\E\left[|\chi_2(Z_H)|^2\right]\right)^{1/4}& &\text{Jensen, $x\mapsto x^2$}\\
        &= \psi_2(H). & &\text{Lemma \ref{lem:minnorm} and (\ref{eqn:normrand})}
    \end{align*}
    Thus, all quantities appearing in the above chain of all inequalities lie within a multiplicative factor of $(1-\gamma)$ of each other. Rearranging and examining the fourth and fifth lines, we obtain the following useful bound: 
    \begin{equation}\label{eqn:secondmomentnormal}
        \E\left[ |Z_H - r|^{-4}\right] \le \frac{1}{(1-\gamma)^4}\left(\E[ |Z_H-r|^{-2}]\right)^{2}.
    \end{equation}
Note that the above inequality is in some sense a ``reverse Jensen" inequality, since the actual Jensen inequality for the function $x\mapsto x^2$ yields $\E[|Z_H-r|^{-4}]\geq \E[|Z_H-r|^{-2}]^2$. So (\ref{eqn:secondmomentnormal}) is articulating that when stagnation happens, the second and fourth moments of $|Z_H-r|^{-1}$ are close to each other, and therefore we should expect concentration for the random variable $|Z_H-r|^{-1}$ (equivalently of $|Z_H-r|$).  Explicitly, we now have
    \begin{align*}
         \P \left[ |Z_H-r|\le \frac{1}{\sqrt{t}}\psi_2(H)\right]
        &= \P\left[ |Z_H-r|^{-2}\ge \frac{t}{\psi_2(H)^2}\right]\\
        &\ge \P\Big[ |Z_H-r|^{-2}\ge t\E[|Z_H-r|^{-2}]\Big] & &\text{since $\tau_{(z-r)^2}(H) \leq \psi_2(H)$}\\
        &\ge (1-t)^2\frac{\E [ |Z_H-r|^{-2}]^2}{\E [|Z_H-r|^{-4}]}& &\text{Paley-Zygmund}\\
        &\ge (1-t)^2(1-\gamma)^4& &\textrm{by \eqref{eqn:secondmomentnormal}},
    \end{align*}
    establishing \eqref{eqn:pz}. 
\end{proof}

Using this lemma we now show that $\Sh_{2,1}$ satisfies its guarantees. The following theorem implies Theorem \ref{thm:mainnormal} by setting $\gamma=0.2$ and calculating $\frac{12\cdot 27}{(4/5)^4}\le 792$.
\begin{theorem} \label{thm:mainnormalgamma}
Let $\gamma\in(0,1)$. The strategy $\Sh_{2,1}$ ensures that
$$\psi_2(\next{H})\le (1-\gamma)\psi_2(H).$$
The worst case complexity of $\Sh_{2,1}$ is $2+|\calS|\le 2+\frac{12\cdot 27}{(1-\gamma)^4}$ calls to $\exactqr(H,(z-(\cdot))^2)$ plus a constant number of arithmetic operations to compute $r_1$ and $r_2$ in Line 1.
\end{theorem}
\begin{proof}
    Suppose Line 2 does not succeed in reducing the potential $\psi_2=\psi_2(H)$ by a factor of $(1-\gamma)$. By Lemma \ref{lem:mainnormal}, for $t\in (0,1)$ to be chosen later:
    $$          \P \left[ |Z_H-r|\le \psi_2/\sqrt{t}\right]
        \ge (1-t)^2(1-\gamma)^4.$$
    Let $\calS$ be an $\epsilon\psi_2$-net of $D(r,\psi_2/\sqrt{t})$ for $\epsilon$ chosen as in Line 3; an elementary packing argument implies that such a net exists with 
    \begin{equation}\label{eq:normalnet}|\calS|\le \frac{\mathrm{area}(D(r,\psi_2/\sqrt{t}))}{\mathrm{area}(D(r,\epsilon\psi_2/2))}=4/t\epsilon^2.
    \end{equation}
    Choose a point $s\in\calS$ satisfying
    $$ \P[Z_H\in D(s,\epsilon\psi_2)]\ge \frac{1}{|\calS|}(1-t)^2(1-\gamma)^2,$$
    which must exist since
    $$D(r,\psi_2/\sqrt{t})\subset \bigcup_{s\in \calS} D(s,\epsilon\psi_2).$$
    We then have
    \begin{align*}
         \tau_{(z-s)^2}^{-4}(H)=\E \left[\frac{1}{|Z_H-s|^4}\right] &\ge \P[|Z_H-s|\le \epsilon \psi_2]\cdot \frac{1}{\epsilon^4\psi_2^4}\\
         &\ge \frac{1}{|\calS|}(1-t)^2(1-\gamma)^4\frac{1}{\epsilon^4\psi_2^4}\\
         &\ge \frac{t\epsilon^2}{4}(1-t)^2(1-\gamma)^4\frac{1}{\epsilon^4\psi_2^4}\\
         &=\frac{(1-\gamma)^4}{27\cdot \epsilon^2\psi_2^4},
    \end{align*}
    choosing $t=1/3$ to maximize the right hand side in the penultimate line. This ultimately yields 
    $$\tau_{(z-s)^2}(H)\leq \frac{27^{1/4}\sqrt{\epsilon}}{1-\gamma} \psi_2(H),$$
which by our  choice of $\epsilon$ in Line 3 and Lemma \ref{lem:lowerboundontau} implies:
$$\psi_2(\exactqr(H,(z-s)^2)\le \tau_{(z-s)^2}(H)\le (1-\gamma)\psi_2(H),$$
as advertised.

The bound on the complexity follows by simply counting the number of calls to $\exactqr$ in Lines $2$ and $4$ and using the estimate $|\calS|\le 4\cdot 3/\epsilon^2$.
\end{proof}

\begin{remark}[Choice of $\gamma$] The decoupling rate $\gamma$ may be viewed as a tuning parameter which trades off the worst case complexity of a single step of the strategy against the worst case total number of steps: if $\gamma$ is larger then each step is guaranteed to make more progress, but the cost of performing exceptional shift is higher as the required $\epsilon-$net is larger. 
\end{remark}

\begin{remark}[Improving the Disk to an Annulus and optimizing Theorem \ref{thm:mainnormalgamma}] 
\label{rem:optimizingmaintheoremgamma}
Inequality (\ref{eqn:pz}) is a tail bound of the random variable $|Z_H-r|$. We note that the other tail can  be controlled too via   Markov's inequality and the upper bound on  \mbox{$\E[|Z_H-r|^{-4}]$} obtained in the proof of Lemma \ref{lem:mainnormal}. Then,  the control on both tails yields that the distribution of $Z_H$ has significant mass on a thin annulus (the inner and outer radii are almost the same) around $r$.

It is instructive to note that when  $\gamma =0$,    this annulus becomes of width $0$, yielding that $Z_H$ is fully supported on a circle with center $r$ and radius $\psi_k(H)$, proving that in this case $\frac{1}{\psi_k(H)}(H-r)$ is in fact a unitary matrix, which recovers Parlett's characterization of fixed points \cite{parlett1966singular} for the Greedy Wilkinson shift. 

In any case, controlling both tails allows one to reduce the search performed by the exceptional shifts from a disk to an annulus, significantly decreasing the size of the net $|\mathcal{S}|$ considered in  Theorem \ref{thm:mainnormalgamma}. Moreover, one can think of taking an optimal net on the relevant region (instead of just using a packing argument to upper bound the size of the net), using higher degree shifts (as discussed in Remark \ref{rem:optimizingmainnormal}), and refining the arguments in Theorem \ref{thm:mainnormalgamma} to further reduce the complexity of the exceptional shift, which may be relevant in practical implementations. We omit such optimization in this paper for the sake of simplicity. 
\end{remark}

\section{Nonnormal Matrices}
\label{sec:mainresult}

In this section we generalize the strategy $\Sh_{2,1}$ for normal matrices to a family $\Sh_{k,B}$ with provably rapid convergence on not necessarily normal matrices satisfying $\kappa_V(H)\le B$ for some given $B\ge 1$, and prove the main Theorem  \ref{thm:main} of this paper. 

It is instructive to note that the only way in which normality was used in the proof of Theorem \ref{thm:mainnormalgamma} is the existence of the spectral measure \eqref{eqn:specmeasurenormal} and corresponding functional calculus \eqref{eqn:normrand}, which crucially implied that $$\E[|q(Z_H)|^{-2}]^{1/2}=\|e_n^*q(H)^{-1}\|$$ for  quadratic polynomials $q$, guiding our choice of shift and ultimately enabling the proof of Lemma \ref{lem:mainnormal}. This fact is no longer true in the nonnormal case. Our main idea is that when $\kappa_V(H)<\infty$, there is an ``approximate functional calculus'' which can be used as an effective substitute, provided that we consider shift polynomials $q$ of appropriately large degree. We will heavily use the following construction throughout this section.
\begin{definition}[Approximate Functional Calculus] \label{def:funccalcrand} Assume that $H = VDV^{-1}$ is diagonalizable, with $V$ chosen\footnote{In the event that there are multiple such choices of $V$ it does not matter which we choose, only that it remains fixed throughout the analysis.} so that $\|V\| = \|V^{-1}\| = \sqrt{\kappa_V(H)}$ and $D$ a diagonal matrix with $D_{i,i} = \lambda_i$, the eigenvalues of $H$. Define $Z_H$ to be the random variable supported on the eigenvalues of $H$ with  distribution 
\begin{equation}
    \label{eqn:funccalcrand}
    \P[Z_H = \lambda_i] =  \frac{|e_n^\ast V e_i|^2}{\|e_n^\ast V\|^2}.
\end{equation}
\end{definition}
Note that $\P[Z_H = \lambda_i] = 1$ exactly when $e_n^\ast$ is a left eigenvector with eigenvalue $\lambda_i$, and that when $H$ is normal, the distribution of $Z_H$ is the spectral measure of $H$ associated to $e_n^*$, so  definition \ref{def:funccalcrand} generalizes \eqref{eqn:normrand}. % following lemma generalizes this fact to the nonnormal case, at a multiplicative cost of $\kappa_V(H)$.

\begin{lemma}
    \label{lem:spectral-measure-apx}
    For any upper Hessenberg $H$ and complex function $f$ whose domain includes the eigenvalues of $H$,
    $$
        \frac{\|e_n^\ast f(H)\|}{\kappa_V(H)} \le \E\left[|f(Z_H)|^2\right]^{\frac{1}{2}} \le \kappa_V(H)\|e_n^\ast f(H)\|.
    $$
\end{lemma}

\begin{proof} 
    By the definition of $Z_H$ above,
    \begin{align*}
        \E\left[|f(Z_H)|^2\right]^{\frac{1}{2}}
        = \frac{\|e_n^\ast f(H) V\|}{\|e_n^*V\|}
        \le \|e_n^\ast f(H)\| \|V\| \|V^{-1}\|
        = \|e_n^\ast f(H)\|\kappa_V(H),
    \end{align*}
    and the left hand inequality is analogous. 
\end{proof}

The upshot of this lemma is that if $q\in\calP_k$, then $\tau_q(H)=\|e_n^*q(H)^{-1}\|^{1/k}$ approximates $\E [|q(Z_H)|^{-2}]^{1/2k}$ up to a factor of $\kappa_V(H
)^{1/k}$, which is close to $1$ if we choose $k\gg \log \kappa_V(H)$.  Thus, by choosing $k$ large enough we can obtain accurate information about $Z_{H}$ by examining the observable quantities $\|e_n^*f(H)\|^{\frac{1}{k}}$, which enables a precise understanding of convergence and a generalization of Lemma \ref{lem:mainnormal} to the nonnormal case. This motivates the use of a higher degree shifting strategy as a way to deal with nonnormality. Since the iterates are all unitarily similar, $\kappa_V$ is preserved with each iteration, so the $k$ required is an invariant of the algorithm. Thus the use of a sufficiently high-degree shifting strategy is both an essential feature and unavoidable cost of our approach.
%Using this lemma with some carefully chosen rational functions $f$ of degree $k$, we are able to probe the distribution of $Z_{H}$ for each iterate $H$ of the algorithm by examining the observable quantities $\|e_n^*f(H)\|^{\frac{1}{k}}$ --- for appropriately large $k$, these are related to $(\E |f(Z_{H})|^2)^{\frac{1}{k}}$ by a multiplicative factor of $\kappa_V(H)^{\frac{1}{k}} \approx 1$, so

In the remainder of this section we describe and analyze the shifting strategy $\Sh_{k,B}$. The proof of the main result appears in Section \ref{sec:shiftingstrategy}. The structure of the proof is similar to that of $\Sh_{2,1}$, but with three important differences: (i) We work with polynomials of degree $k\approx \log B \log\log B$ rather than $2$ in light of the above discussion. This requires settling for approximate (i.e., $\theta-$optimal) rather than exact Ritz values when $k\ge 5$, even in exact arithmetic.  (ii) The ``greedy'' choice in $\Sh_{2,1}$ cannot be made exactly due to the absence of \eqref{eqn:normrand}. We introduce the notion of a ``promising Ritz value'' (Section \ref{sec:almostmonotonicity}) as an approximate surrogate for the ``greedy'' choice with similar properties, and describe an efficient procedure for finding such a Ritz value (Section \ref{sec:promising}). (iii) All of the proofs involve carrying around approximation factors arising from the use of Lemma \ref{lem:spectral-measure-apx}, $\theta$-optimality, and promising Ritz values. The required exceptional shift (analyzed in Section \ref{sec:exceptional}) is correspondingly larger. \\

\noindent{\em Notation and Constants.} $\K \ge \kappa_V(H)$ denotes an upper bound on its eigenvector condition number and and $k \ge 2$ a power of two, which the reader may consider for concreteness to be on the order of $\log \K \log\log\K$; all logarithms will be taken base two for simplicity. 
%Fixing some $\gamma \in (0,1)$, we will show that our shifting strategy guarantees \textit{potential reduction}: the efficient computation of a Hessenberg matrix $\next{H}$, unitarily equivalent to $H$, with the property that
%\begin{equation}
%    \label{eq:potentialreduction}
 %   \pot(\next{H})\leq \gamma \psi_k(H).
 %\end{equation}
Note
that the relationship \eqref{eqn:kappasetting} between $k$ and $B$ is {\em not} required for the
proof of potential reduction, but impacts the cost of performing
each iteration. The table below collates notation and constants which will appear
throughout this section. 

\begin{table}[h]
    \centering
    \begin{tabular}{l|l|l}
        \textbf{Symbol} & \textbf{Meaning} & \textbf{Typical Scale} \\
        \hline
        $H$ & Upper Hessenberg matrix & \\
        $Z_H$ & Random variable in Definition \ref{def:funccalcrand} & \\
        $\K$ & Eigenvector condition bound & $\K \ge \kappa_V(H)$ \\
        $k$ & Shift degree & $O(\log \K \log\log\K)$ \\
        $\delta$ & Decoupling parameter &   \\
        $\gamma$ & Decoupling rate & $0.2$ \\
        $\r$ & Approximation parameter for Ritz values & $2$ \\
        $\cp$ & Promising Ritz value parameter & $B^{4k^{-1} \log k} = 1 + o(1)$
    \end{tabular}
\end{table}

\renewcommand{\r}{\theta}

\subsection{Promising Ritz Values and Almost Monotonicity of the Potential}
\label{sec:almostmonotonicity}
%The first step of our shifting strategy is to single out Ritz values with the following special property.

In the same spirit as Wilkinson's shift, which chooses a particular Ritz value (out of two), but using a different criterion, our shifting strategy will begin by choosing a Ritz value (out of $k$) that has the following property for some $\alpha\geq 1$. This is a generalization of the ``greedy'' Wilkinson shift considered in Section \ref{sec:normal}. \begin{definition}[$\alpha$-promising Ritz value]
Let $\cp \ge 1$,  $\calR = \{r_1,...,r_k\}$ be a set of $\r$-approximate Ritz values for $H$, and $p(z) = \prod_{i=1}^k (z - r_i)$. We say that $r \in \calR$ is  {\em $\cp$-promising} if
\begin{equation}\label{eqn:defpromising}     \E \left[ \frac{1}{|Z_H-r|^k}\right] \ge  
    \frac{1}{\cp^k} \E\left[ \frac{1}{|p(Z_H)|}\right].
\end{equation}
\end{definition}
Note that there is at least one $1$-promising Ritz value in every set of approximate Ritz values, since
\begin{equation}\label{eqn:amgm}
    \frac{1}{k} \sum_{i=1}^k\E \left[\frac{1}{|Z_H-r_i|^k}\right]=\E \left[\frac{1}{k}\sum_{i=1}^k  \frac{1}{|Z_H-r_i|^k}\right] \ge \E \left[ \frac{1}{|p(Z_H)|}\right]
\end{equation}
by linearity of expectation and AM/GM. The notion of $\cp$-promising Ritz value is a relaxation which can be computed efficiently from the entries of $H$ (in fact, as we will explain in Section \ref{sec:promising}, using a small number of implicit QR steps with Francis-like shifts of degree $k/2$).

As a warm-up for the analysis of the shifting strategy, we will first show
that if $k\gg \log \kappa_V(H)$ and $r$ is a promising Ritz value, the
potential is \emph{almost monotone} under the shift $(z - r)^k$. This
articulates the phenomenon observed in Example \eqref{eqn:exampleB} and suggests that
promising Ritz values should give rise to good polynomial shifts. Monotonicity is not
actually used in the proof of our main theorem, which instead relies on the closely related property (\ref{eq:keyineq}) established below.

\begin{lemma}[Almost-monotonicity and Moment Comparison]
\label{lem:almostmonotonicity}
    Let $\calR = \{r_1, \dots, r_k\}$ be a set of $\theta$-optimal Ritz values, as in Definition \ref{def:approxritz}, and assume that $r\in \calR$ is $\alpha$-promising. If $\next{H} = \exactqr(H,(z-r)^k )$ then
    $$\pot(\next H) \le \kappa_V(H)^{\frac{2}{k}}\alpha\theta \pot(H),$$
    and moreover
    \begin{equation}
    \label{eq:keyineq}
         \E\left[|Z_H - r|^{-2k}\right] \geq \E\left[|Z_H - r|^{-k}\right]^2 \geq  \frac{1}{\kappa_V(H)^{2}(\cp \theta \pot(H))^{2k} }. 
    \end{equation}
\end{lemma}

\begin{proof} Let $p(z) = \prod_{i=1}^k (z - r_i)$. The claim follows from the following chain of inequalities: 
      \begin{align}
      \sqrt{\E\left[|Z_H - r|^{-2k}\right]}
      &\ge \E\left[|Z_H - r|^{-k}\right] & &\text{Jensen, $x\mapsto x^2$} \nonumber \\
        &\ge\frac{1}{\cp^k}\E[|p(Z_H)|^{-1}] & & \text{$r$ is $\cp$-promising}\nonumber\\
        &\ge\frac{1}{\cp^k}\frac{1}{{\E[|p(Z_H)|]}} & & \text{Jensen, $x\mapsto 1/x,x>0$} \nonumber\\
        &\ge\frac{1}{\cp^k}\frac{1}{\sqrt{\E[|p(Z_H)|^2]}} & & \text{Jensen, $x\mapsto x^2$} \nonumber\\
        & \ge\frac{1}{\cp^k}\frac{1}{\|e_n^\ast p(H)\|\kappa_V(H)} & & \text{Lemma \ref{lem:spectral-measure-apx}}\nonumber \\
        & \ge\frac{1}{\cp^k}\frac{1}{\theta^k\|e_n^\ast \chi_k(H)\|\kappa_V(H)} & & \text{Definition \ref{def:approxritz} of $\theta$-optimal}\nonumber\\
        &=\frac{1}{\cp^k}\frac{1}{\theta^k\psi_k(H)^k \kappa_V(H)} & & \text{Lemma \ref{lem:minnorm}}.\nonumber
    \end{align}
This already shows (\ref{eq:keyineq}). For the other claim, rearrange both extremes of the above inequality to get
\begin{align*}
    \cp \theta \kappa_V(H)^{\frac{1}{k}} \pot(H) 
    &\geq  \E\left[|Z_H - r|^{-2k}\right]^{-\frac{1}{2k}}
    \\ 
    &\geq \frac{\tau_{(z-r)^k}(H)}{\kappa_V(H)^{\frac{1}{k}}}  & & \text{Lemma \ref{lem:spectral-measure-apx}}
    \\ & \geq \frac{\pot(\next{H})}{\kappa_V(H)^{\frac{1}{k}}} & & \text{Lemma \ref{lem:lowerboundontau}}
\end{align*}
which concludes the proof. 
\end{proof}

In Section \ref{sec:shiftingstrategy}, we will see that when the shift
associated with a promising Ritz value does not reduce the potential, Lemma
\ref{lem:almostmonotonicity} can be used to provide a two-sided bound on the
quantities $\E[|Z_H-r|^{-2k}]$ and $\E[|Z_H-r|^{-k}]^2$.  This is the main
ingredient needed to obtain information about the distribution of $Z_H$ when
potential reduction is not achieved.

\subsection{The Shifting Strategy}
\label{sec:shiftingstrategy}
In this section we specify the shifting strategy $\Sh_{k,B}$ and prove Theorem \ref{thm:main}. An important component of our shifting scheme, presented in detail in Section \ref{sec:promising}, is a simple subroutine, ``$\find$,'' guaranteed to produce an $\cp$-promising Ritz value with $\cp = \kappa_V(H)^{4 k^{-1} \log k}$. Guarantees for this subroutine are stated in the lemma below and proved in Section \ref{sec:promising}. 

\begin{lemma}[Guarantees for $\find$]
    \label{lem:find}
    The subroutine $\find$ specified in Section \ref{sec:promising} produces a $\kappa_V(H)^{4 k^{-1}\log k}$-promising Ritz value, using at most $12k\log k n^2 + \log k$ arithmetic operations.
\end{lemma}

Our strategy is then built around the following dichotomy, which crucially uses the $\cp$-promising property: in the event that a degree $k$ implicit QR step with the $\cp$-promising Ritz value output by $\find$ does {\em not} achieve potential reduction, we show that there is a modestly sized set of exceptional shifts, one of which is {\em guaranteed} to achieve potential reduction. These exceptional shifts are constructed by the procedure ``$\exc$'' described in Section \ref{sec:exceptional}. The overall strategy is specified below.\\

\noindent 
\begin{boxedminipage}{\textwidth}
$$\Sh_{k, \K}$$
    \textbf{Input:} Hessenberg $H$ and a set $\calR$ of $\r$-approximate Ritz values of $H$ \\
    \textbf{Output:} Hessenberg $\next{H}$, which will be the next matrix in the iteration.\\
    \textbf{Requires:} $0 < \pot(H)$ and  $\kappa_V(H) \le \K$\\
    \textbf{Ensures:} $\pot(\next{H})\le (1-\gamma)\psi_k(H)$ and $\kappa_V(\next H) \le \K$
    \begin{enumerate}
        \item \label{line:sh1} $r \gets \find(H,\calR)$
        \item \label{line:sh2} If $\pot(\exactqr(H,(z-r)^k)) \le (1-\gamma) \pot(H)$, output $\next H = \exactqr(H,(z-r)^k)$
        \item \label{line:sh3} Else, $\calS \gets \exc(H,r,\K)$
        \item \label{line:sh4} For each $s \in \calS$, if $\pot(\exactqr(H,(z - s)^k)) \le (1-\gamma)\pot(H)$, output $\next H = \exactqr(H,(z-s)^k)$
    \end{enumerate}
\end{boxedminipage}\\

The failure of line \eqref{line:sh2}
of $\Sh_{k, \K}$ to reduce the potential gives useful quantitative information about
the distribution of $Z_H$, articulated in the following lemma.  This will then be used to design the set $\calS$ of
exceptional shifts produced by $\exc$ in line \eqref{line:sh3} and prove that
at least one of them makes progress in line \eqref{line:sh4}. 
\begin{lemma}[Stagnation Implies Support]
    \label{lem:main}
    Let $\gamma \in (0,1)$ and $\r \ge 1$, and let $\calR = \{r_1, \dots, r_k\}$ be a set of $\r$-approximate Ritz values of $H$. Suppose $r \in \calR$ is $\cp$-promising and assume 
    
\begin{equation}
    \label{eqn:jstag} \psi_k\left(\exactqr(H,(z - r)^k)\right)\geq (1-\gamma)\psi_k(H)>0.
\end{equation}
Then $Z_H$ is well-supported on an disk of radius approximately $\cp\psi_k(H)$ centered at $r$ in the following sense: for every $t\in (0,1)$:
    \begin{equation}
        \label{eqn:pz2}
        \P \left[ |Z_H-r|\le \r\cp\left(\frac{\kappa_V(H)}{t}\right)^{\frac{1}{k}}\pot(H)\right]
        \ge (1-t)^2\frac{(1-\gamma)^{2k}}{\cp^{2k}\r^{2k}\kappa_V(H)^{4}}.
    \end{equation}
\end{lemma}

\begin{proof} Observe that $H-r$ is invertible since otherwise, for $\next{H}= \exactqr(H,(z - r)^k)$, we would have $\psi_k(\next{H})=0$ by Lemma \ref{lem:lowerboundontau}.
   Our assumption implies that that:
    \begin{align*}
        (1-\gamma)\psi_k(H) &\le \psi_k(\next{H}) & &\text{hypothesis}\\
        &\le \tau_{(z - r)^k}(H) & &\text{Lemma \ref{lem:lowerboundontau}}\\
        &= \|e_n^\ast(H-r)^{-k}\|^{-\frac{1}{k}} & &\text{definition}\\
        &\le \left(\frac{\kappa_V(H)}{\E\left[|Z_H - r|^{-2k}\right]^{\frac{1}{2}}}\right)^{1/k} & &\text{Lemma \ref{lem:spectral-measure-apx}}.
    \end{align*}
    Rearranging and using (\ref{eq:keyineq}) from Lemma \ref{lem:almostmonotonicity} we get
    \begin{equation}
    \label{eqn:firstmoment}
        \frac{\kappa_V(H)^{2}}{(1-\gamma)^{2k}\psi_k(H)^{2k}}\ge \E\left[|Z_H - r|^{-2k}\right] \ge \E\left[|Z_H - r|^{-k}\right]^2 \ge \frac{1}{\cp^{2k}\theta^{2k}\psi_k(H)^{2k} \kappa_V(H)^{2}},
    \end{equation}
    which upon further rearrangement yields the ``reverse Jensen'' type bound (note that for the function $x\mapsto x^2$,  Jensen's inequality yields the complementary $\E[|Z_H-r|^{-k}]^2 \leq \E[|Z_H-r|^{-2k}]$ ):
    \begin{equation}\label{eqn:secondmoment}
        \frac{\E[|Z_H - r|^{-2k}]}{\E[|Z_H - r|^{-k}]^2}\le \left(\frac{\cp\theta}{(1-\gamma)}\right)^{2k}\kappa_V(H)^{4}.
    \end{equation}
    We now have
    \begin{align*}
         \P \left[ |Z_H-r|\le \frac{\cp}{t^{1/k}}\theta\psi_k(H)\kappa_V^{1/k}\right]
        &= \P\left[ |Z_H-r|^{-k}\ge t\frac1{\cp^k\theta^k\psi_k(H)^k\kappa_V}\right]\\
        &\ge \P\left[ |Z_H-r|^{-k}\ge t\E[|Z_H-r|^{-k}]\right] & &\text{by \eqref{eqn:firstmoment}}\\
        &\ge (1-t)^2\frac{\E [ |Z_H-r|^{-k}]^2}{\E [|Z_H-r|^{-2k}]}& &\text{Paley-Zygmund}\\
        &\ge (1-t)^2\frac{(1-\gamma)^{2k}}{\cp^{2k}\theta^{2k}\kappa_V(H)^{4}}& &\textrm{by \eqref{eqn:secondmoment}},
    \end{align*}
    establishing \eqref{eqn:pz2}, as desired.
\end{proof}

In Section \ref{sec:exceptional}, we will use Lemma \ref{lem:main} to prove the following guarantee on $\exc$.
  
\begin{lemma}[Guarantees for $\exc$]
    \label{lem:exc}
    The subroutine $\exc$ specified in Section \ref{sec:exceptional} produces a set $\calS$ of exceptional shifts, one of which achieves potential reduction. If $\r \le 2$, $\gamma = 0.2$, and $\cp = \K^{4\log k /k}$, then both the number of degree $k$ $\exactqr$ calls required for $\exc$, and the size of $\calS$, are at most
    $$
        \tnet\left(0.002\K^{-\frac{8\log k + 4}{k}}\right),
    $$
    where $\tnet(\epsilon) = O(\epsilon^{-2})$ denotes number of points in an efficiently computable $\epsilon$-net of the unit disk. In the normal case, taking $\K = \cp = \r = 1$, $k = 4$, $\gamma = 0.2$, the arithmetic operations required and the size of $|\calS|$ are both bounded by $50$.
\end{lemma}

We are now ready to prove Theorem \ref{thm:main}. 

\begin{proof}[Proof of Theorem \ref{thm:main}] {\em Rapid convergence.}
    In the event that we choose a $\cp$-promising Ritz value in step \eqref{line:sh1} that does not achieve potential reduction in step \eqref{line:sh2}, Lemma \ref{lem:exc} then guarantees we achieve potential reduction in \eqref{line:sh3}. Thus each iteration decreases the potential by a factor of at least $(1-\gamma)$, and since $\pot(H_0) \leq \|H\|$ we need at most $$\frac{\log(1/\delta)}{\log(1/(1-\gamma))}\le 4\log(1/\delta)$$ iterations before $\psi_k(H_t)\le \delta\|H_0\|$, which in particular implies $\delta$-decoupling.\\
    
    \noindent {\em Arithmetic Complexity.} Computing a full set $\calR$ of $\r$-approximate Ritz values of $H$ has a cost $\tapr(k, \r,\delta)$. Then, using an efficient implicit QR algorithm (cf. Definition \ref{def:stableiqr}) each computation of \mbox{$\exactqr(H, (z-r_i)^k)$} has a cost of $7kn^2$. By Lemma \ref{lem:find}, we can produce a promising Ritz value in at most $12 k \log k n^2 + \log k$ arithmetic operations. Then, in the event that the promising shift fails to reduce the potential the algorithm calls $\exc$, which takes $\tnet(0.002\K^{-\frac{8\log k + 4}{k-1}})$ arithmetic operations to specify the set $\calS$ of exceptional shifts. Some exceptional shift achieves potential reduction, and we pay $7kn^2$ operations for each one that we check. 
\end{proof}

\subsection{Efficiently Finding a Promising Ritz Value}
\label{sec:promising}
In this section we show how to efficiently find a promising Ritz value, in $O(n^2k\log k)$ arithmetic operations. Note that it is trivial to find a $\kappa_V(H)^{2/k}$-promising Ritz value in $O(n^2k^2)$ arithmetic operations simply by computing $\|e_n^*(H-r_i)^{-k/2}\|$ for $i=1,\ldots, k$ with $k$ calls to $\exactqr(H,(z-r_i)^{k/2})$, choosing the maximizing index $i$, and appealing to Lemma \ref{lem:spectral-measure-apx}. The content of Lemma \ref{lem:find} below that this can be done considerably more efficiently if we use a binary search type procedure. This improvement has nothing to do with the dynamical properties of our shifting strategy so readers uninterested in computational efficiency may skip this section.\\

\noindent \begin{boxedminipage}{\textwidth}
    $$\find$$
    \textbf{Input:} Hessenberg $H$, a set $\calR=\{r_1,\ldots,r_k\}$ of $\theta$-optimal Ritz values of $H$. \\
    \textbf{Output:} A complex number $r\in \calR$ with suitable properties that will be used in the main shift.  \\
    \textbf{Requires:} $\psi_k(H)>0$\\
    \textbf{Ensures:} $r$ is $\cp$-promising for $\cp = \kappa_V(H)^{\frac{4\log k}{k}}$.
    \begin{enumerate}
        \item For $j = 1,...,\log k$
        \begin{enumerate}
            \item Evenly partition $\calR = \calR_0 \sqcup \calR_1$, and for $b = 0,1$ set $p_{j,b} = \prod_{r \in \calR_{b}}(z - r)$
            \item $\calR \gets \calR_b$, where $b$ maximizes $\|e_n^\ast p_{j,b}(H)^{-2^{j-1}}\|$
        \end{enumerate}
        \item Output $\calR = \{r\}$
    \end{enumerate}
\end{boxedminipage}
 
\begin{proof}[Proof of Lemma \ref{lem:find} (Guarantees for $\find$)]
    First, observe that $\|e_n^*q(H)\|\neq 0$ for every polynomial appearing in the definition of $\find$, since otherwise we would have $\psi_k(H)=0$.
    
    On the first step of the subroutine $p_{1,0}p_{1,1} = p$, the polynomial whose roots are the full set of approximate Ritz values, so 
    \begin{align*}
        \max_b \|e_n^\ast p_{1,b}(H)^{-1}\| 
            &\ge \frac{1}{\kappa_V(H)^2}\E\left[\frac{1}{2}\left( |p_{1,0}(Z_H)|^{-2} + |p_{1,1}(Z_H)|^{-2}\right)\right] & & \text{Lemma \ref{lem:spectral-measure-apx}} \\
            &\ge \frac{1}{\kappa_V(H)^2}\E[|p(Z_H)|^{-1}] & & \text{AM/GM}.   
    \end{align*}
    On each subsequent step, we've arranged things so that $p_{j+1,0}p_{j+1,1} = p_{j,b}$, where $b$ maximizes $\|e_n^\ast p_{j,b}(H)^{-2^{j-1}}\|$, and so by the same argument
    \begin{align*}
        \max_b \|e_n^\ast p_{j+1,b}(H)^{-2^{j}}\|^2 
        &\ge \frac{1}{\kappa_V(H)^2}\E\left[\frac{1}{2}\left(|p_{j+1,0}(Z_H)|^{-2^{j+1}} + |p_{j+1,1}(Z_H)|^{-2^{j+1}}\right)\right] & & \text{Lemma \ref{lem:spectral-measure-apx}} \\
        &\ge \frac{1}{\kappa_V(H)^2}\E\left[|p_{j+1,0}(Z_H) p_{j+1,1}(Z_H)|^{-2^{j}}\right] & & \text{AM/GM} \\
        &\ge \frac{1}{\kappa_V(H)^4} \|e_n^\ast (p_{j+1,0}(H)p_{j+1,1}(H))^{-2^{j-1}}\| & & \text{Lemma \ref{lem:spectral-measure-apx}} \\
        &= \frac{1}{\kappa_V(H)^4}\max_b \|e_n^\ast p_{j,b}(H)^{-2^{j-1}}\|.
    \end{align*}
    Paying a further $\kappa_V(H)^2$ on the final step to convert the norm into an expectation, we get
    $$
        \E\left[|Z_H - r|^{-k}\right] \ge \frac{1}{\kappa_V(H)^{4\log k}}\E\left[|p(Z_H)|^{-1}\right]
    $$
    as promised.
    
    For the runtime, we can compute every $\|e_n^\ast p_{j,b}(H)^{-2^{j-1}}\|$ by running an implicit QR step with the polynomials $p_{j,b}^{2^{j-1}}$, all of which have degree $k/2$. There are $2\log k$ such computations throughout the subroutine, and each one requires $6 k n^2$ arithmetic operations. Beyond that we need only compare the two norms on each of the $\log k$ steps.
\end{proof}

\begin{remark}[Opportunism and Judicious Partitioning]
In practice, it may be beneficial to implement $\find$ {\em opportunistically}, meaning that in each iteration one should check if the new set of Ritz values gives potential reduction (this can be combined  with the computation of $\|e_n^* p_{j, b}(H)^{-2^{j-1}}\|$ and implemented with no extra cost). Moreover, note that $\find$ does not specify a way to partition the set of Ritz values obtained after each iteration, and as can be seen from the above proof, the algorithm
works regardless of the partitioning choices.  It is conceivable that a judicious choice of the partitioning could be used to obtain further improvements.  \end{remark}

\subsection{Analysis of the Exceptional Shift}
\label{sec:exceptional}

To conclude our analysis, it remains only to define the subroutine ``$\exc$,'' which produces a set $\calS$ of possible exceptional shifts in the event that an $\cp$-promising Ritz value does not achieve potential reduction. The main geometric intuition is captured in the case when $H$ is normal and $\kappa_V(H) = 1$. Here, $\find$ gives us a $1$-promising Ritz value $r$ and Lemma \ref{lem:main} with $t = 1/2$ tells us that if $r$ does not achieve potential reduction, than $Z_H$ has measure at least $\tfrac{1}{4}((1-\gamma)/\r)^{2k}$ on a disk of radius $R:=2^{1/k}\r\pot(H)$.

For any $\epsilon > 0$, we can easily construct an $R\epsilon$-\textit{net} $\calS$ contained in this disk --- i.e., a set with the property that every point in the disk is at least $R\epsilon$-close to a point in $\calS$ --- with $O(1/\epsilon)^2$ points. One can then find a point $s \in \calS$ satisfying
\begin{align*}
    \tau_{(z-s)^k}(H)^{-2k} & = \|e_n^\ast(H - s)^{-k}\|^2 
   \\ & = \E[|Z_H - s|^{-2k}] 
  \\ & \ge \frac{\P[|Z_H - s| \le \pot(H)]}{|\calS| (R\epsilon)^{2k}} 
\\ & \approx \frac{1}{4}\left(\frac{(1-\gamma)}{\theta}\right)^{2k}\frac{1}{R^{2k}\epsilon^{2k-2}},
\end{align*}
where the first equality is by normality of $H$, and second inequality comes from choosing $s \in \calS$ to maximize $|Z_H - s|^{-2k}$. Since $\pot(\exactqr(H,(z-s)^k)) \le \tau_{(z-s)^k}(H)$, we can ensure potential reduction by setting $\epsilon \approx \frac{(1-\gamma)^2 R}{\r \pot(H)} \approx ((1-\gamma)/\r)^2$.

When $H$ is nonnormal, the chain of inequalities above hold only up to factors of $\kappa_V(H)$, and $\find$ is only guaranteed to produce a $\kappa_V(H)^{4 \log k / k}$-promising Ritz value. The necessary adjustments are addressed below in the implementation of $\exc$ and the subsequent proof of its guarantees.\\

\noindent \begin{boxedminipage}{\textwidth}
$$\exc$$
\textbf{Input:} Hessenberg $H$, a $\r$-approximate Ritz value $r$, a condition number bound $\K$, promising parameter $\cp$ \\
\textbf{Output:} A set $\calS \subset \C$, which will be used as a set of exceptional shifts \\
\textbf{Requires:} $\kappa_V(H) \le \K$, $r$ is $\cp$-promising, and $\pot(\exactqr(H,(z -r)^k) \ge (1-\gamma)\pot(H)$ \\
\textbf{Ensures:} For some $s \in \calS$, $\pot(\exactqr(H,(z-s)^k) \le (1-\gamma)\pot(H)$
\begin{enumerate}
    \item \label{line:exc1} $R \gets 2^{1/k}\r\cp \K^{1/k}\pot(H)$
    \item \label{line:exc2} $\epsilon \gets \left(\frac{(1-\gamma)^2}{(12\K^4)^{1/k}\cp^2\r^2}\right)^{\frac{k}{k-1}}$
    \item $\calS \gets \epsilon R$-net of $R\pot(H)$.
\end{enumerate}
\end{boxedminipage}

\begin{proof}[Proof of Lemma \ref{lem:exc}: Guarantees for $\exc$]
    Instantiating $t = 1/2$ in equation \eqref{eqn:pz2}, we find that for the setting of $R$ in line \eqref{line:exc1} of $\exc$,
    $$
        \P\left[|Z_H - r| \le D(r,R)\right] \ge \frac{1}{4\K^4}\left(\frac{(1-\gamma)}{\cp\r}\right)^{2k}.
    $$
    Let $\calS$ be an $\epsilon R$-net of $D(r,R)$; it is routine that such a net has at most $(1 + 2/\epsilon)^2 \le 9/\epsilon^2$ points. By Lemma \ref{lem:lowerboundontau}, to show that some $s \in \calS$ achieves potential reduction, it suffices to find one for which
    $$
        \|e_n^\ast(H - s)^{-k}\|^2 \ge \frac{1}{(1-\gamma)^{2k}\pot(H)^{2k}}.
    $$
    We thus compute
    \begin{align*}
        \max_{s \in \calS}\|e_n^\ast(H - s)^{-k}\|^2 
        &\ge \frac{1}{\kappa_V(H)^2 |\calS|}\sum_{s \in \calS}\E\left[|Z_H - s|^{-2k}\right] \\
        &\ge \frac{\epsilon^2}{9 \K^2}\E\left[\sum_{s \in \calS}|Z_H - s|^{-2k} \cdot \textbf{1}_{Z_H \in D(r,R)}\right] & & \text{Fubini and $\kappa_V(H) \le B$} \\
        &\ge \frac{\epsilon^2}{9 \K^2}\E\left[\max_{s \in \calS}|Z_H - s|^{-2k} \cdot \textbf{1}_{Z_H \in D(r,R)}\right] \\
        &\ge \frac{\epsilon^2}{9 B^2}\E\left[\frac{\textbf{1}_{Z_H \in D(r,R)}}{(\epsilon R)^{2k}}\right] & & \text{$\calS$ is an $\epsilon R$-net}\\
        &\ge \frac{\P[Z_H \in D(r,R)]}{9 \K^2 R^{2k}\epsilon^{2k-2}} \\
        &\ge \frac{1}{(1-\gamma)^{2k}\psi(H)^{2k}}
    \end{align*}
    with the second to last line following from the fact that some $s \in \calS$ is at least $\epsilon R$-close to $Z_H$ whenever the latter is in $D(r,R)$, and the final inequality holding provided that 
    \begin{align*}
        \epsilon \le \left(\frac{\P\big[|Z_H - r| \le R\pot(H)\big](1-\gamma)^{2k}\pot(H)^{2k}}{9B^2R^{2k}}\right)^{\frac{1}{2k-2}}.
    \end{align*}
    Expanding the probability and using the definition of $R$ in line \ref{line:sh1}, it suffices to set $\epsilon$ smaller than
    $$
        \left(\frac{(1-\gamma)^{2k}}{4\K^4\cp^{2k}\r^{2k}} \cdot \frac{(1-\gamma)^{2k}\pot(H)^{2k}}{9\K^2} \cdot \frac{1}{4\K^2\cp^{2k}\r^{2k}\pot(H)^{2k}}\right)^{\frac{1}{2k-2}} 
        = \left(\frac{(1-\gamma)^2}{(12\K^4)^{1/k}\cp^2\r^2}\right)^{\frac{k}{k-1}},
    $$
    which is the quantity appearing in line \ref{line:sh2}. Setting $\r = 2$, $\gamma = 0.2$, and $\cp = \K^{4\log k / k}$, and using $k \ge 2$, we obtain the expression appearing in $\tnet(\cdot)$ in the statement of Lemma \ref{lem:exc}.
    
    However, a more practical choice (and the one that we will use in the companion paper \cite{banks2021global2}) is an equilateral triangular lattice with spacing $\sqrt 3 \epsilon$, intersected with the $D(r,(1 + \epsilon)R)$. Such a construction is optimal as $\epsilon \to 0$, and can be used to give a better bound on $\tnet(\epsilon)$ when $\epsilon$ is small. For instance, by adapting an argument of \cite[Lemma 2.6]{armentano2018stable} one can show that this choice of $\calS$ satisfies
    $$
        \tnet(\epsilon) \le \frac{2\pi}{3\sqrt 3}(1 + 1/\epsilon)^2 + \frac{4\sqrt 2}{\sqrt 3}(1 + 1/\epsilon) + 1.
    $$
    In the normal case, when $\K = \cp = \r = 1$, $k= 4$, and $\gamma = 0.2$, the above bound gives
    $$
        |\calS| \le \tnet\left(\left(\frac{0.8^2}{12^{1/4}}\right)^{4/3}\right) \le 49.9.
    $$
\end{proof}

\subsubsection*{Acknowledgments}
We thank Jim Demmel, Daniel Kressner, and Cleve Moler for helpful conversations and references. We thank the anonymous referees for a thorough reading and feedback which greatly improved the paper.

\bibliographystyle{alpha}
\bibliography{ShiftedQR}

\end{document}